\documentclass[12pt]{article}

\usepackage[a4paper, margin=2cm]{geometry}

\usepackage{amssymb}
\usepackage{amsmath,amsthm}
\usepackage{epstopdf}
\usepackage{subfigure}
\usepackage{tikz}
\usepackage{tkz-graph}
\usepackage{mathtools}
\usepackage{csquotes}
\usepackage[bb=boondox]{mathalfa} 

\newtheorem{theorem}{Theorem}[section]
\newtheorem{corollary}[theorem]{Corollary}
\newtheorem{lemma}[theorem]{Lemma}
\newtheorem{proposition}[theorem]{Proposition}

\newtheorem{example}{Example}

\def\qed{\vbox{\hrule
  \hbox{\vrule\hbox to 5pt{\vbox to 8pt{\vfil}\hfil}\vrule}\hrule}}

\def\endproof{\unskip \nobreak \hskip0pt plus 1fill \qquad \qed \par \vspace{0.15cm}}

\newcommand{\mb}{\mathbb}

\begin{document}

\title{Combinatorial Fiedler Theory and Graph Partition}

\author{
 Enide Andrade\footnote{Center for Research and Development in Mathematics and Applications,
 Department of Mathematics,  University of Aveiro, Portugal. {\tt enide@ua.pt}} \\
 Geir Dahl\footnote{Department of Mathematics,   University of Oslo, Norway. {\tt geird@math.uio.no.} Corresponding author.}
 }

\maketitle


\date{}


\begin{abstract}
 Partition problems in graphs are extremely important in applications, as shown in  the Data science and  Machine learning literature. One approach  is spectral partitioning based on a  Fiedler vector, i.e., an eigenvector corresponding to the second smallest eigenvalue $a(G)$ of the Laplacian matrix $L_G$ of the graph $G$. This problem corresponds to the  minimization of a   quadratic form associated with $L_G$, under  certain constraints involving the $\ell_2$-norm. We introduce and investigate a similar problem, but using the $\ell_1$-norm to measure distances. This leads to a new parameter $b(G)$ as the optimal value. We show that a well-known cut problem arises in this approach, namely the sparsest cut problem. We prove connectivity results and different bounds on this new parameter, relate to Fiedler theory and show explicit expressions for $b(G)$ for trees. We also comment on an $\ell_{\infty}$-norm version of the problem.
\end{abstract}



\section{Introduction} 
\label{sec: introduction}

Often real world networks contain clusters, that is, groups of points each with a large number of neighbors  among them and not many connections to the outside.  In Data science and Machine learning the task of {\em clustering} is very important. Given a set of data points (in some space) and their common properties measured in terms of distances, the clustering problem consists in finding subsets of these data points that are ``similar''. If the points to be clustered are vertices in a graph, and the edges connecting these vertices are the only information available, then the problem is called {\em graph clustering} \cite{Luxburg,Schaeffer}. There are many methods for graph clustering, and one popular such method is {\em spectral clustering}.   The basis is then spectral bisection where  a Fiedler vector \cite{Fiedler2}   is used for partitioning a graph into two connected subgraphs based  on the sign of the components of the vector. This splitting may be repeated for each of the parts and thereby obtain a desired partition. 

 Consider an unweighted (undirected) simple graph $G=(V,E)$ and let $n=|V|$. Throughout the paper we assume that $G$ is connected. Recall that the {\em Laplacian matrix} $L_G$ is the $n \times n$ matrix $L_G=D_G-A_G$ where $D_G$ is the diagonal matrix with the vertex degrees on the diagonal, and $A_G$ is  the adjacency matrix of $G$. 
$L_G$ is positive semidefinite and therefore it has only real, nonnegative eigenvalues. The algebraic connectivity $a(G)$  is the second smallest eigenvalue of  $L_G$, and it is known as a connectivity measure in the graph \cite{Fiedler_alg_conn,Molitierno11}. In particular, $G$ is connected if and only if $a(G)>0$. The matrix $L_G$ is singular, 0 is the smallest eigenvalue and a corresponding eigenvector is the all ones vector $e$. Therefore, by the Courant-Fischer theorem \cite{HornJohnson13}, 
\[
     a(G)=\min\{x^TL_G \,x: e^Tx=0, \; \|x\|_2=1\}.
\]
Here $\|x\|_2=(\sum_i x_i^2)^{1/2}$ is the (Euclidean) $\ell_2$-norm of $x =(x_1, x_2, \ldots, x_n) \in \mb{R}^n$. By using a standard factorization $L_G=B^TB$ where $B$ is the edge-vertex incidence matrix of $G$, we have the alternative expression
\begin{equation}
 \label{eq:smooth2}
     a(G)=\min\{\sum_{uv \in E} (x_u-x_v)^2: \sum_{v \in V} x_v=0, \; \|x\|_2=1\}.
\end{equation}
Thus, an eigenvector corresponding to the eigenvalue $a(G)$, usually called a {\em Fiedler vector}, can be seen as assigning  values to the vertices to obtain an optimal ``smoothing'' along edges, i.e., small difference between end points of edges, under the two normalization constraints, see \cite{Spielman10}. 
These constraints assure that we avoid a constant solution $x=\lambda e$ for some $\lambda$ and, also, the norm constraint avoids scaling (to get similar solutions) and the zero vector. We shall therefore call  the minimization problem in (\ref{eq:smooth2})  the {\em $\ell_2$-graph smoothing} problem. 
Here $\ell_2$ refers to the  fact that both the (objective) function $\sum_{uv \in E} (x_u-x_v)^2$ and the norm constraint involve the Euclidean norm $\ell_2$. 

The motivating question of our study  is:

\begin{itemize}
 \item What happens if we modify the optimization problem (\ref{eq:smooth2}) by changing the norm involved to  the $\ell_1$-norm (the sum norm)?
\end{itemize}
The  $\ell_1$-norm of $x =(x_1, x_2, \ldots, x_n) \in \mb{R}^n$ is defined by $\|x\|_1=\sum_i |x_i|$.
In fact,   often in mathematics different norms may be used in the study of some (approximation) problem. A well-known such example is the linear approximation problem $\min_{x \in \mb{R}^n} \|Cx-b\|$ where $C$ is an $m \times n$ (real) matrix, $b \in \mb{R}^m$ and $\| \cdot\|$ is some vector norm. When the norm is $\ell_2$ we obtain the least squares problem, and for the norms $\ell_1$ and $\ell_{\infty}$ one may use linear programming to solve the problem.  Therefore, it is important to understand the properties of solutions, and how they depend on the choice of norm.\\

Therefore, our main goal is to consider a new graph smoothing problem which we call {\em $\ell_1$-graph smoothing}. It is similar to $\ell_2$-graph smoothing except that we change the norm from the $\ell_2$-norm to the $\ell_1$-norm. Let $G=(V,E)$ be a given graph with at least one edge. 
The {\em $\ell_1$-graph smoothing} problem is the following optimization problem
\begin{equation}
 \label{eq:smooth1}
     b(G)=\min\{\sum_{uv \in E} |x_u-x_v|: \sum_{v \in V} x_v=0, \; \|x\|_1=1\}.
\end{equation}
Clearly the minimum here is attained by some $x$ as the constraints define a compact set and the function to be minimized is continuous. Note that the constraint set is not a convex set. An optimal solution $x$ in (\ref{eq:smooth1}) will be called an {\em $\ell_1$-Fiedler vector}. Then $x$ satisfies  $\sum_v x_v=0$, $\|x\|_1=1$ and $\sum_{uv \in E} |x_u-x_v|=b(G)$.

A main contribution  of our paper is indicated in  Figure \ref{fig:contribution} (where $x \perp e$ means $e^Tx=0$). 
The two graph smoothing problems are indicated in the two last columns of the figure. One sees how the two problems are quite similar. Based on several intermediate results we establish that optimal solutions in the  new problem correspond to  so-called sparsest cuts. This also means that strong connections, with bounds, between the two optimal values $a(G)$ and $b(G)$ may be found. Moreover, there are important consequences in terms of computational complexity.  A main  contribution of this paper is to show that {\em there is a very natural optimization approach that is underlying sparsest cuts}. 
Thus the problem (\ref{eq:smooth1}) can be handled using a {\em combinatorial approach}. For related combinatorial approaches to Perron values of trees, see \cite{EALCGD19,EAGD17}. We will discuss the computational complexity for these graph smoothing problems, although this is not done in any detail. However, some remarks are given on the complexity of approximation problems related to $b(G)$, and we believe more theoretical work can be done here. The approaches discussed in \cite{CiardoZivny23} for approximate graph coloring may be of interest in this connection.

\medskip

\begin{figure}[ht]
\begin{center}
\begin{tabular}{|l||c|c|} \hline 
    {\tt Problem:} & \mbox{\rm $\ell_2$-graph smoothing}   & \mbox{\rm $\ell_1$-graph smoothing} \\ \hline
    {\tt Norm:} & $\ell_2$-norm  & $\ell_1$-norm \\  \hline 
    {\tt Minimize:} & $\sum_{uv \in E} (x_u-x_v)^2$   & $\sum_{uv \in E} |x_u-x_v|$    \\ \hline 
    {\tt Constraints:} & $\|x\|_2=1$,  $x \perp e$ & $\|x\|_1=1$,  $x \perp e$  \\ \hline 
    {\tt Solution:} & \mbox{\rm Fiedler vector, partition}   & \mbox{\rm $\ell_1$-Fiedler vector, sparsest cut} \\ \hline
    {\tt Optimal value:} & \mbox{\rm algebraic connectivity $a(G)$}   & \mbox{\rm $b(G)$ } \\ \hline   
    {\tt Property:} & \mbox{\rm connected parts}   & \mbox{\rm connected parts} \\ \hline
    {\tt Approach:} & \mbox{\rm spectral} & \mbox{\rm combinatorial} \\  \hline 
\end{tabular}
\end{center}
\caption{Graph smoothing  and partition problems.}
\label{fig:contribution}
\end{figure}

The remaining part of this Introduction is devoted to some more results from spectral partitioning and relevant spectral graph theory. Finally, an overview of the next sections is given.

In \cite{Urschel} one studies the maximal error in spectral bisection where the two parts in the partition have the same size.  In \cite{Kim}  the authors investigate graphs having Fiedler vectors with unbalanced sign patterns such that a partition can result in two connected subgraphs that are distinctly different in size. They also characterize graphs with a Fiedler vector having exactly one negative component. Motivated by these results  we recall  some facts concerning spectral partitioning and  Cheeger bounds. For a much deeper discussion of these areas, we refer to the lectures notes  \cite{Spielman19} and  \cite{Trevisan17}, and the references found there.
Consider again a graph $G=(V,E)$ and let $n=|V|$. Let $S \subseteq V$ be a nonempty vertex set, where $S \not = V$.  Let $\delta(S)$ denote the set of edges $uv$ where $u \in S$ and $v \not \in S$, this is called the {\em cut} induced by $S$. Cuts are important objects in graph theory, combinatorial optimization as well as in applications. Define the {\em relative cut size} 
\begin{equation}
 \label{eq:reative-cut-size}
      \xi(S)=\frac{|\delta(S)|}{|S|},
\end{equation}
which is the size of the cut relative to the size of the vertex set $S$. This is an important notion  in this paper. 
The {\em isoperimetric number} of $G$, also called {\em Cheeger's constant}, is the parameter
\[
    i(G)=\min_S \xi(S),
\]
where the minimum is taken over all nonempty subsets $S$ of $V$ with $|S| \le \lfloor n/2 \rfloor$.\\


 A basic treatment of the isoperimetric number and its properties may be found in \cite{Mohar2}.
A related notion \cite{Mohar} is the {\em edge density} of a cut, defined as follows,  
\[
 \rho(S) = \frac{ |\delta(S)|}{ |S||V\backslash {S}|}.
 \]
  This concept represents the density of the edges in $G$ between the set $S$ and its complement, compared to the number of edges in a  complete bipartite graph (with vertex set $S$ and its complement). This parameter $\rho(S)$ is also called the {\em sparsity} of the cut. A {\em sparsest cut} is a cut which minimizes $\rho(S).$
 For more literature related to these concepts, see \cite{Hoory}.
  
  A small calculation shows the following relation between  edge density and relative cut sizes for each subset $S$ (with $\emptyset \subset S \subset V$) 
  \begin{equation}
 \label{eq:xi-rho}
      \xi(S)+\xi(V \setminus S)=n \rho(S).
  \end{equation}
Below we give some inequalities relating edge density and Laplacian eigenvalues, in particular the algebraic connectivity.

\begin{theorem}[\cite{Mohar}] 
\label{thm:Mohar}
 
Let $G$ be a graph of order $n$. For any  nonempty subset $S$ of vertices of $G$,  $S \neq V$, the edge density is uniformly bounded below and above by 
$$\frac{a(G) }{ n}\leq \rho(S) \leq \frac{\lambda_{1}}{n},$$
where $\lambda_1$ is the largest eigenvalue of $L(G).$
\end{theorem}

 In \cite{Fallat_Kirkland_Pati} one characterized the graphs for which $a(G) = \rho(S),$ for some subset $S$ of vertices. 

There is another upper bound on the minimal density of cuts in terms of $a(G).$

\begin{theorem} [\cite{Mohar}] 
 Let $G=(V,E)$ be a graph of order $n$ with at least two edges. Then
$$\min\{ \rho(S): S \subset V, S \neq \emptyset \} \leq \frac{2}{ n} \sqrt{ a(G) [ 2 d_{\rm{max}}(G) -a(G)]},$$
where $d_{\rm{max}}(G)$ is the maximal vertex degree in $G$. 
\end{theorem}

This upper bound is a strong discrete version of the well-known {\em Cheeger's inequality} from differential geometry \cite{Cheeger}, bounding the first eigenvalue of a Riemannian manifold. It appeared in \cite{Alon, AlonandMilman} and later, as an improved edge version in \cite{Mohar2}.

\medskip
The remaining paper is organized as follows. In Section \ref{sec:ell1-smoothing} we study the set of  feasible solutions of the $\ell_1$-graph smoothing problem and present a rewriting of the problem. The main results are then presented in Section \ref{sec:main}, where it is shown that optimal solutions correspond to sparsest cuts with a connectivity property. The computational complexity is also settled. In Section \ref{sec:compare} a comparison of $b(G)$ and other parameters is made. Section \ref{sec:special} is devoted to examples and specific classes of graphs where  explicit expressions for $b(G)$ are found. Moreover, a computational example is shown. In the final section we briefly consider a graph smoothing problem based on  the $\ell_{\infty}$-norm.

\medskip
{\bf Notation:} Vectors in $\mb{R}^n$ are considered as column vectors and identified with the real $n$-tuples. The $i$'th component of a vector $x \in \mb{R}^n$ is usually denoted by $x_i$ ($ i \le n$). A zero matrix, or vector, is denoted by $O$, and an all ones vector is denoted by $e$ (the dimension should be clear from the context).
For a real number $c$ define $c^+=\max\{c,0\}$ and $c^-=\max\{-c,0\}$. Then $c=c^+-c^-$ and $|c|=c^++c^-$. 
Let $x=(x_1, x_2, \ldots, x_n) \in \mb{R}^n$. Define $x^+=(x^+_1, x^+_2, \ldots, x^+_n) \in \mb{R}^n$ and $x^-=(x^-_1, x^-_2, \ldots, x^-_n) \in \mb{R}^n$. Thus, $x=x^+-x^-$. 
$K_n$ represents the complete graph. Moreover, $P_n$ (resp. $S_n$) is the path (resp. the star) with $n$ vertices.

\section{The $\ell_1$-graph smoothing problem}
 \label{sec:ell1-smoothing}

As mentioned, our main goal is to introduce and study  the $\ell_1$-graph smoothing problem. In this section we shall rewrite the problem into a convenient form. 

Throughout, we let the  vertex set of the graph $G$ be $V=\{v_1, v_2, \ldots, v_n\}$, and we identify a function $x \in \mb{R}^V$ with the vector $x=(x_1, x_2, \ldots, x_n)$ where $x_j=x(v_j)$ for each $j\le n$. Define $F_1$ as the {\em feasible set} in (\ref{eq:smooth1}), i.e., 
\[
     F_1=\{x \in \mb{R}^n: \sum_j x_j=0, \; \|x\|_1=1\}.
\]
We also  define the ``smoothing function'' 
$
       f_1(x)=\sum_{uv \in E} |x_u-x_v|.
$
Therefore,  the  $\ell_1$-graph smoothing problem is to minimize $f_1(x)$ subject to $x \in F_1$.\\

\begin{example}
{\rm 
Consider $P_4=v_1,v_2,v_3,v_4$, the path with four vertices. Let $x^1=(-1/2, 0, 0, 1/2)$, so the only nonzeros are in  the end vertices. Then $x^1 \in F_1$ and $f_1(x^1)=1/2+0+1/2=1$. Next, consider $x^2=(-1/4, -1/4, 1/4, 1/4)$. Then $x^2 \in F_1$ and $f_1(x^2)=0+1/2+0=1/2$. So, $x^2$ is better than $x^1$ and $b(P_4)\le 1/2$. But is $x^2$ optimal? The answer is yes, as will follow from later results. We remark that the algebraic connectivity of $P_4$ is given by $a(P_4)=0.5858$.
\endproof
}
\end{example} 

An alternative description of the feasible set $F_1$ is given next. 
\begin{lemma}
 \label{lem:feasible set} 
 \begin{equation}
  \label{eq:F-new}
     F_1=\{x \in \mb{R}^n: \sum_{j: x_j\ge 0} x_j=1/2, \; \sum_{j: x_j\le 0} x_j=-1/2\}.
 \end{equation}
\end{lemma}
\begin{proof}
 Let $x \in F_1$. Then  
 \[
  1=\|x\|_1=\sum_{j: x_j\ge 0} x_j - \sum_{j: x_j\le 0} x_j, 
 \]
 and adding this equation to  $\sum_j x_j=0$ gives $2\sum_{j: x_j\ge 0} x_j=1$, i.e., 
 $\sum_{j: x_j\ge 0} x_j=1/2$. This implies $\sum_{j: x_j\le 0} x_j=-1/2$ as $\sum_j x_j=0$. 
  Conversely,  if $\sum_{j: x_j\ge 0} x_j=1/2$ and $\sum_{j: x_j\le 0} x_j=-1/2$, then $\sum_j x_j=0$ and $\|x\|_1=1$, by the equation above.
\end{proof}

A vector $x \in F_1$ will be called a {\em feasible} solution of (\ref{eq:smooth1}). Thus, by Lemma \ref{lem:feasible set}, a feasible solution partitions  the vertices into three subsets depending on the sign  of each $x_j$, $\pm 1$ or 0, and the sum of the components of $x$ in the positive and negative part is the same in absolute value, $\sum_{j: x_j>0} x_j = | \sum_{j: x_j<0} x_j|=1/2$. 
\medskip

We next rewrite problem (\ref{eq:smooth1}). Define  
\begin{equation}
 \label{eq:smooth1-new}
     \beta(G)=\min\{\sum_{ij \in E} |x^1_i-x^2_i-x^1_j+x^2_j|: 
     \sum_{j=1}^n x^1_j=1/2, \; \sum_{j=1}^n x^2_j=1/2, \; (x^1)^Tx^2=0, \; x^1,  x^2\ge O \}
\end{equation}
where $x^1=(x^1_1, x^1_2, \ldots, x^1_n)$ and $x^2=(x^2_1, x^2_2, \ldots, x^2_n)$ are vectors in $\mb{R}^n$.
Note that the constraints assure that for each $i$ at least one of the two variables $x^1_i$ and $x^2_i$ is zero. 
The next result connects the two optimization problems $(\ref{eq:smooth1})$ and $(\ref{eq:smooth1-new})$.

\begin{lemma}
 \label{lem:rewrite} 
 The following holds:
 
 $(i)$ If $x$ is optimal in  $(\ref{eq:smooth1})$, then $x^1=x^+$, $x^2=x^-$ is optimal in $(\ref{eq:smooth1-new})$. 
 
 $(ii)$ If $x^1,x^2$ are optimal in  $(\ref{eq:smooth1-new})$, then $x=x^1-x^2$ is optimal in $(\ref{eq:smooth1})$.
 
 $(iii)$  $b(G)=\beta(G)$. 

\end{lemma}
\begin{proof}
 (i) 
 Properties (i) and (ii) follows by replacing each variable $x_j$ by $x_j=x^1_j-x^2_j$ where $x^1_j$ and $x^2_j$ are two nonnegative variables.   In this construction $x^1$ and $x^2$ are only unique up to a positive additive constant in each term, but the orthogonality constraint $(x^1)^Tx^2$ assures uniqueness and that $x^1=x^+$ and $x^2=x^-$. Thus $(\ref{eq:smooth1-new})$ is a reformulation of $(\ref{eq:smooth1})$. This implies that the optimal values coincide, so (iii) holds.
\end{proof}

Later we prove that it is {\em NP}-hard to compute $b(G)$ and a corresponding $\ell_1$-Fiedler vector. Still, the previous lemma  means that computing $b(G)$, and the corresponding optimal $x$, may be done by solving $(\ref{eq:smooth1-new})$. This is a problem of minimizing a piecewise linear convex function subject to linear constraints and a ``complementarity constraint'' saying that $x^1_ix^2_i=0$ for each $i$. 
This problem can be written as a linear programming problem with certain linear complementarity constraints corresponding to the orthogonality $x^1 \perp x^2$, as explained next.  Consider the following optimization problem with variables $x^1_j, x^2_j$ ($j \le n$) and $y_{ij}$ for $ij \in E$. 

\begin{equation}
 \label{eq:LP-smooth}
 \begin{array}{lrll} \vspace{0.1cm}
 \mbox{\rm minimize}  & \sum_{ij \in E} y_{ij} \\ \vspace{0.1cm}
    &       x^1_i-x^2_i-x^1_j+x^2_j &\le y_{ij}  & (ij \in E),\\ \vspace{0.1cm}
     &      -x^1_i+x^2_i+x^1_j-x^2_j &\le y_{ij} & (ij \in E),\\ \vspace{0.1cm}
     & \sum_{j=1}^n x^1_j&=1/2, \; \\ \vspace{0.1cm}
      & \sum_{j=1}^n x^2_j&=1/2, \; \\
      & x^1_j x^2_j&=0 & (j \le n), \\
     & \multicolumn{2}{c}{x^1 \ge O, \; x^2\ge O, \; y \in \mb{R}^E.}
 \end{array}
\end{equation}
In fact, the first two constraints are equivalent to 
$
            |x^1_i-x^2_i-x^1_j+x^2_j| \le y_{ij} \;\;(ij \in E)
$
and due to the minimization equality must hold here for every $ij \in E$. As mentioned, the  $\ell_1$-graph smoothing  problem is {\em NP-hard}, but several general integer programming based algorithms have been developed that may be used to give approximate solutions of $(\ref{eq:LP-smooth})$.
A basic reference on the linear complementarity problem is \cite{CottleDantzig68}. We leave it as an interesting  idea for further research to use this formulation in order to find approximate solutions of  the $\ell_1$-graph smoothing problem.
In the final section of this paper we also use a related linear programming approach to a graph smoothing problem based on  the $\ell_{\infty}$-norm.

\section{$\ell_1$-graph smoothing and sparsest cuts}
\label{sec:main}

In this section we investigate the $\ell_1$-graph smoothing problem closer and establish strong properties of the optimal solutions. This leads to a connection to sparsest cuts.
Recall that we assume that the graph  $G$ is connected. The proof gives a construction based on the relative cut size $\xi(S)$ defined in (\ref{eq:reative-cut-size}).

\begin{theorem}
 \label{thm:conn} 
 
  Let $x$ be an $\ell_1$-Fiedler vector with the maximum number of zeros. Then   the subgraph induced by $\{v \in V: x_v >0\}$ is connected and  the subgraph induced by $\{v \in V: x_v <0\}$ is connected.  
\end{theorem}
\begin{proof}
We shall first prove that   the subgraph induced by $V^+:=\{v \in V: x_v >0\}$ is connected. 

The proof is by contradiction, so assume the subgraph induced by $V^+$ is {\em not} connected. Then there must exist disjoint subsets $S_1$ and $S_2$ of $V^+$ such that  
\begin{description}
\item (i) no edge joins $S_1$ and $S_2$, 

\item (ii)  $x_v >0$ $(v \in S_1 \cup S_2)$, 

\item (iii) if $vw \in \delta(S_1)$ with $v \in S_1$ (so $w \not \in S_1$) then $x_w \le 0$. 
\end{description}

We discuss different cases. 

{\em Case $1$: $\xi(S_1)>\xi(S_2)$.} 
Let $\epsilon$ be a ``suitably small'' positive number; in fact, $\epsilon<\min\{x_v: v \in V^+\}$ works. 
Define $y \in \mb{R}^V$ based on $x$ as follows:
\[
y_v=
\left\{
\begin{array}{ll}
   x_v-\epsilon & (v \in S_1), \\ 
   x _v+(|S_1|/|S_2|)\epsilon &(v \in S_2), \\
   x_v &(\mbox{\rm otherwise}). 
\end{array}
\right.
\]
Then 
\[
  \sum_{v:y_v \ge 0}  y_v -\sum_{v:x_v \ge 0}  x_v=\sum_{v \in S_1} (-\epsilon) + \sum_{v \in S_2}(|S_1|/|S_2|)\epsilon 
  = -|S_1| \epsilon + |S_1| \epsilon=0.
\]
Therefore 
\[
     \sum_{v:y_v \ge 0}  y_v = \sum_{v:x_v \ge 0}  x_v =1/2. 
\]
Moreover, 
\[
    \sum_{v:y_v \le 0}  y_v = \sum_{v:x_v \le 0}  x_v =-1/2;
\]
recall that $y_{v}= x_{v},$ for all $v \not \in (S_1 \cup S_2)$. This proves that  $y \in F_1$.

Next, for each edge  $vw \in \delta(S_1)$ with $v \in S_1$, $|y_v-y_w|=(x_v-x_w)-\epsilon$ as $x_v >0 \ge x_w$. Similarly, for each edge  $vw \in \delta(S_2)$ with $v \in S_2$, $|y_v-y_w|=(x_v-x_w)+(|S_1|/|S_2|)\epsilon$. Therefore 
\begin{equation}
 \label{eq:f-y-x}
 \begin{array}{ll} \vspace{0.2cm}
   f_1(y)-f_1(x)&=-\epsilon |\delta(S_1)| +(|S_1|/|S_2|)\epsilon  |\delta(S_2)| \\ \vspace{0.2cm}
   &=-\epsilon \big(|\delta(S_1)| -(|S_1|/|S_2|) |\delta(S_2)|\big) \\ \vspace{0.2cm}
   &< 0
 \end{array}  
\end{equation}
as $\xi(S_1)>\xi(S_2)$ means $|\delta(S_1)|/|S_1|>|\delta(S_2)|/|S_2|$ so 
$|\delta(S_1)| -(|S_1|/|S_2|) |\delta(S_2)|>0$. This proves that $b(G) \le f_1(y)<f_1(x)$, contradicting that $x$ is an $\ell_1$-Fiedler vector. 

{\em Case $2$: $\xi(S_1)<\xi(S_2)$.} By symmetry of $S_1$ and $S_2$ this can be treated by similar arguments and a contradiction is derived. 

{\em Case $3$: $\xi(S_1)=\xi(S_2)$.} We use the same construction of the vector $y$ as in Case 1, but we let $\epsilon=\min\{x_v: v \in S_1\}$. Then $y \in F_1$, and from (\ref{eq:f-y-x}) we see that 
$f_1(y)=f_1(x)$, so $y$ is also an $\ell_1$-Fiedler vector. However, $y$ has at least one more zero than $x$, and this contradicts the choice of $x$ (initially in the proof).

Thus, in each case we obtained a contradiction, which proves that the subgraph induced by $V^+:=\{v \in V: x_v >0\}$ is connected. The proof that the subgraph induced by $\{v \in V: x_v <0\}$ is connected is completely similar, by choosing $\epsilon < \min \{ | x_v |, v \in V^{-}\}$ in Case $1$ and 
$\epsilon =\min  \{ | v_v|, v \in S_1 \}$ in Case $3$.
\end{proof}

Next we give  a main result which shows an explicit formula for $b(G)$ which is of a combinatorial nature. This will give a strong connection to the notions presented in the Introduction.  

\medskip
We say that the pair $(S_1, S_2)$ is a \textit{quasi-bipartition} of $V$ if 
\[
   S_1, S_2 \subset V, \;S_1  \cap S_2= \emptyset, \; \mbox{\rm and} \; S_1, S_2 \neq \emptyset.
\]
For a quasi-bipartition $(S_1, S_2)$ let $x=x^{(S_1,S_2)}=(x_v: v \in V) \in \mb{R}^{n}$ be the vector with

\[
x_{v} =
\left\{
\begin{array}{rl} \vspace{0.1cm}
 \frac{1}{2 | S_1 |} & ( v \in S_1), \\ \vspace{0.1cm}
-\frac{1}{2 | S_2 |} & ( v \in S_2), \\ \vspace{0.1cm}
0                                       & \mbox{(otherwise).}
 \end{array}\right.
\]

\begin{theorem}
 \label{thm:main} 
  For any graph $G$ 
  \begin{equation}
   \label{eq:alpha_exp}
      b(G) =
      \frac{1}{2}\min \{\xi(S_1)+\xi(S_2): \mbox{\rm $(S_1, S_2)$ is a quasi-bipartition}     \}.
  \end{equation}
   Moreover, when $(S_1, S_2)$ is optimal in $(\ref{eq:alpha_exp})$, the corresponding vector $x^{(S_1,S_2)}$ is an  $\ell_1$-Fiedler vector $x$.
\end{theorem}
\begin{proof}
 Let $x$ be an $\ell_1$-Fiedler vector.  Let $\kappa(x)$ be the number of distinct positive elements in $x$ (i.e., the cardinality of the set of positive components). Choose $x$ (optimal) with $\kappa(x)$ smallest possible.  Define 
 \[
  S^+=\{v: x_v>0\}, \; S^-=\{v: x_v<0\} \;\; \mbox{\rm and} \;\;  S^0=\{v: x_v=0\}. 
\]
  Both $S^+$ and $S^-$ are nonempty. Also define $M_1=\max_v x_v$,  $M_2=\min\{ x_v: x_v>0\}$ and 
 \[
      S^+_1=\{v: x_v=M_1\},  \; S^+_*=\{v: M_2<x_v<M_1\}, \;\; \mbox{\rm and} \;\; S^+_2=\{v: x_v=M_2\}. 
 \]
  We now prove that there is an optimal solution of (\ref{eq:smooth1})  where all positive $x_v$'s are equal. If $M_1=M_2$, there is nothing to prove, so assume $M_1>M_2$. Let $\epsilon_1$ be a ``small'' number in absolute value 
 and let $\epsilon_2$ satisfy $|S^+_1|\epsilon_1=|S^+_2|\epsilon_2$, i.e., 
$\epsilon_2=\epsilon_1|S^+_1|/|S^+_2|$.
Define $x^{\epsilon}$ by
\[
   x^{\epsilon}_v = 
   \left\{
   \begin{array}{cl} \vspace{0.1cm}
      M_1-\epsilon_1 & (v \in S^+_1), \\
      M_2+\epsilon_2 & (v \in S^+_2), \\
      x_v & ({\rm otherwise}). \\
   \end{array}
   \right.
\]
Observe that the relationship between $\epsilon_1$ and $\epsilon_2$ assures that 
\[
   \sum_{v: x^{\epsilon}_{v} \geq 0}x^{\epsilon}_v=1/2 \; \mbox{\rm and} \;
   \sum_{v: x^{\epsilon}_{v} \leq 0}x^{\epsilon}_{v}=-1/2. 
\]
 Thus, $x^{\epsilon} \in F_1$ provided that $\epsilon_1$ is small enough in absolute value. 

Let $\Delta(\epsilon)=f_1(x^{\epsilon}) - f_1(x)$. Then $\Delta(\epsilon)= \sum_{uv \in E}\Delta_{uv}$  where $\Delta_{uv}=|x_{u}^{\epsilon }- x_{v}^\epsilon |-|x_{u} -x_{v} |$ is given as follows for each edge $uv$

(a)  if $u \in S^+_1$, $v \not \in (S^+_1 \cup S^+_2)$, then $\Delta_{uv}=-\epsilon_1$, \smallskip

(b)  if $u \in S^+_1$, $v \in S^+_2$, then $\Delta_{uv}=-\epsilon_1-\epsilon_2$,\smallskip

(c)  if $u \in S^+_2$, $v  \in S^+_*$, then $\Delta_{uv}=-\epsilon_2$,\smallskip

(d)  if $u \in S^+_2$, $v  \in S^- \cup S_0$, then $\Delta_{uv}=\epsilon_2$,\smallskip

\noindent and for all other edges $\Delta_{uv}=0$. 
Let $N_a$, $N_b$, $N_c$, $N_d$ be the number of edges in categories a, b, c, d, respectively. Then
\[
   \Delta(\epsilon)=N_a(-\epsilon_1) + N_b(-\epsilon_1 -\epsilon_2)+N_c(-\epsilon_2)+N_d(\epsilon_2).
\]
By inserting the expression for $\epsilon_2$ above we obtain
\[
 \Delta(\epsilon_1)=\eta \epsilon_1,
\]
for some number $\eta$ that depends on $N_a$, $N_b$, $N_c$, $N_d$,  $|S^+_1|$ and $|S^+_2|$. Moreover, there exists an $\epsilon^*>0$ such that for all $\epsilon_1$ with $|\epsilon_1|< \epsilon^*$, the vector $x^{\epsilon}$ lies in $F_1$.  We must have $\eta=0$, otherwise we could let $\epsilon_1$ be small enough and with opposite sign as $\eta$ and then
\[
   0>\eta \epsilon_1=\Delta(\epsilon_1)=f_1(x^{\epsilon}) - f_1(x).
\]
So $f_1(x^{\epsilon}) < f_1(x)$ which contradicts the optimality of $x$. Therefore, $\eta=0$ and  
\[
   f_1(x^{\epsilon}) = f_1(x).
\]
Now, let $\epsilon_1>0$ and increase $\epsilon_1$ until $\Delta_{uv}$ becomes 0 for some edge for which it was previously positive. This happens if either the smallest value in $S^+$ has been decreased to 0, or when largest value has been decreased to the second largest, or the smallest value has been increased to the second smallest (or both of these occur simultaneously). This  $x^{\epsilon_1}$  is also an $\ell_1$-Fiedler vector. As $\kappa(x^{\epsilon_1})< \kappa(x)$, this contradicts our choice of $x$. Thus, by contradiction, it follows that $M_1=M_2$, so all positive components in $x$ are equal. 

Finally, among all $\ell_1$-Fiedler vectors whose positive components coincide with that of $x$, we proceed to treat the negative components in exactly the same manner as the first part of the proof. As a result, we find an $\ell_1$-Fiedler vector where all the negative components have the same value, and all the positive components have the same value. Let $x$ denote this vector and define $S_1=\{v: x_v>0\}$, $S_2=\{v: x_v<0\}$ and $S_0=\{v: x_v=0\}$. Let $[S_i : S_j]$ be the set of edges $uv$ such that $u\in S_i, v\in S_j$, where $i,j \in\{0,1,2\}$. Then $x_v=1/(2|S_1|)$ for all $v \in S_1$ and $x_v=-1/(2|S_2|)$ for all $v \in S_2$. Moreover, %
\[
\begin{array}{ll} \vspace{0.2cm}
 b(G)&= f_1(x) \\ \vspace{0.2cm}
    &= \sum_{uv \in E} |x_u-x_v| \\ \vspace{0.2cm}
    &=\sum_{uv \in \delta(S_1) \cup \delta(S_2)}  |x_u-x_v| \\ \vspace{0.2cm}
    &= \sum_{uv: u \in S_1, v\in S_2}  |x_u-x_v|+\sum_{uv: u \in S_1, v\in S_0}  |x_u-x_v|+\sum_{uv: u \in S_2, v\in S_0}  |x_u-x_v|  \\ \vspace{0.2cm}
    &=  | [S_1:S_2] | ( (1/(2|S_1|) - (-1/(2|S_2|))+ | [S_1:S_0] |  (1/(2|S_1|))+ | [S_2:S_0] |  (-1/(2|S_2|)) \\ \vspace{0.2cm}
    &=( | [S_1:S_2] |+ | [S_1:S_0] |)( (1/(2|S_1|) - ( | [S_1:S_2] |+ | [S_2:S_0] |)( -(1/(2|S_2|))  \\ \vspace{0.2cm}
    &= (1/2)\big(|\delta(S_1)|/|S_1| + |\delta(S_2)|/|S_2|\big) \vspace{0.2cm} \\ \vspace{0.2cm}
    &=(1/2)\big(\xi(S_1)+\xi(S_2)\big).
\end{array}
\]
From this calculation we also see that for any $S'_1, S'_2 \subset V$ with $S'_1 \cap S'_2 = \emptyset$, $S'_1, S'_2 \not = \emptyset$, there exists an $x' \in F_1$ with $f_1(x' )=(1/2)\big(\xi(S'_1)+\xi(S'_2))$, and therefore $f_1(x') \ge b(G)$.  
This proves the theorem.
\end{proof}

 Let $ \Gamma$ be the set of all quasi-partitions of $V.$ Thus 

\[
   b(G) = (1/2) \min_{(S_1, S_2) \in \Gamma } \big(\xi(S_{1}) + \xi(S_{2})\big).
\]

The next corollary says that if $x$ is an $\ell_1$-Fiedler vector then $x$ has no component equal to zero. 
\begin{corollary} 
\label{cor:no-zero}
 Let $x=(x_1, x_2, \ldots, x_n)$ be an $\ell_1$-Fiedler vector. Then $x_i \not = 0$ for all $i \le n$.
\end{corollary}

\begin{proof} 
Choose an  $\ell_1$-Fieldler vector vector $x$.  Partition the set of vertices into $S^+=\{v: x_v >0\},$   $S^-=\{v: x_v <0\}$ and $S^{0}= \{ v: x_{v}=0\}.$ We recall that  $S^+$ and  $S^-$ are nonempty and we want to prove that $S^{0}= \emptyset$.  Consider the following edge sets
\begin{eqnarray*}
E_1   &  = & \{ uv: u \in S^{+}, v \in S^0\}, \\
E_2   & = & \{ uv: u \in S^{+}, v \in S^{-}\}, \\
E_{3}& = & \{ uv: u \in S^{-}, v \in S^0\}.
\end{eqnarray*}
Then
\[
\begin{array}{lll}\vspace{0.2cm}
b(G) &= \frac{1}{2}\big(\xi(S^{+}) + \xi (S^{-})\big) \\ \vspace{0.2cm}
        &= \frac{1}{2}\big(\frac{| \delta(S^{+})| }{| S^{+} |} +\frac{| \delta(S^{-})| }{| S^{-} |}\big)\\ \vspace{0.2cm}
                &= \frac{1}{2}\big(\frac{ | E_1 |+ | E_2 |} {| S^{+} |}+ \frac{ | E_2 |+ | E_3 |} {| S^{-} |}\big).
\end{array}
\]
Assume  that $ S^{0} \neq \emptyset$. 
\begin{itemize}
\item Let  $S' = S^{+} \cup S^0$ and consider the quasi-bipartition $(S',S^{-})$.  
Then
\[
  \xi(S') + \xi (S^{-}) =\frac{ |E_2|+ |E_3|} {|S^{+}|+ |S^0| }+\frac{ |E_2|+ |E_3 |} {|S^{-}|}. 
\]
If $|E_3| \leq  | E_1|$,  then $\frac{1}{2}(\xi(S') + \xi (S^{-}) )< \frac{1}{2}(\xi(S^{+}) + \xi (S^{-})), $ which is not possible by the definition of $b(G).$ Therefore $|E_3| >  | E_1|$.

\item Let  $S'' = S^{-} \cup S^0$  and consider the quasi-bipartition $ (S^+,S'')$. Then

\[
  \xi (S^{+}) + \xi(S'')  =\frac{ | E_1 |+ | E_2  |} {| S^{+}|} + \frac{ | E_1 |+ | E_2  |} {| S^{-} |+ | S^0 | }. 
\]

If $| E_1 | \leq  | E_3 |$,  then $\frac{1}{2}(\xi(S^+) + \xi (S'') )<\frac{1}{2}(\xi(S^{+}) + \xi (S^{-})),$ which is not possible by definition of $b(G).$ Therefore $|E_3| <  | E_1|$.
\end{itemize}
Thus, $|E_3| <  | E_1| < |E_3|$; a contradiction. We conclude that $ S^{0} = \emptyset$, as desired. 
\end{proof}

The next theorem sums up the results above. It  connects $b(G)$ to the minimum edge density and also shows that we may restrict to connected subgraphs when computing $b(G)$. 

\begin{theorem}
 \label{thm:main2} 
  For any graph $G$ 
  \begin{equation}
   \label{eq:alpha_exp2}
      b(G) =
      \frac{n}{2}\min_S \rho(S)
  \end{equation}
  where $\rho(S)=\frac{ |\delta(S)| }{|S|(n-|S|)}$ is the edge density of the cut $\delta(S)$ and the minimum is taken for nonempty subsets $S$ of $V$ such that $S\not = V$ and both $S$ and its complement induce connected subgraphs of $G$. 
\end{theorem}
\begin{proof}
 This follows by combining  Theorem \ref{thm:main} and Corollary \ref{cor:no-zero}:  we have that the quasi-bipartition $(S_1,S_2)$ must be a partition, i.e., $S_1\cup S_2=V$ as from Corollary \ref{cor:no-zero} any $\ell_1$-Fiedler vector has no components equal to zero. Therefore $\delta(S_1)=\delta(S_2)$. The argument in the proof of Theorem \ref{thm:conn} assures connectedness of the two subgraphs. The definition of edge density of a cut and the relation in  (\ref{eq:xi-rho}) give the desired formula for $b(G)$. 
\end{proof}

Thus we arrive at  the important insight: 
\begin{itemize}
   \item {\em up to a multiplicative constant, namely $n/2$, the optimal value $b(G)$ in the $\ell_1$-graph smoothing problem coincides with the smallest edge density of a cut in $G$.}
\end{itemize}

This gives a very intuitive interpretation of the partitioning of a graph $G$ according to the positive and negative values in the $\ell_1$-Fiedler vector: it corresponds to a cut of smallest edge density, also called a {\em sparsest cut}. In fact, for such a sparsest cut $\delta(S)$ a corresponding $\ell_1$-Fiedler vector is $x=x^S=(x_v: v \in V)$ given by
\begin{equation}
 \label{eq:sparsest-Fiedler}
  x_v= \left\{
 \begin{array}{rr}\vspace{0.1cm}
    \frac{1}{2|S|} & (v \in S) \\ 
    -\frac{1}{2|\tilde{S}|} & (v \in \tilde{S})     
 \end{array}
 \right.
\end{equation}
where $\tilde{S}=V \setminus S$. Thus, this correspondence is underlying when we later refer to a solution of the $\ell_1$-graph smoothing problem or the sparsest cut problem, i.e., one solution can be converted into the other.

\begin{example}
{\rm 
Consider the path $P_5=v_1,v_2,v_3,v_4,v_5$. By Theorem \ref{thm:main2}, including the connectivity result, it is easy to see that  a sparsest cut is $\delta(S^*)$ where $S^*=\{v_1,v_2\}$.  So $\min_S \rho(S)=1/6$ and $b(G) = 5/12$.
} \endproof
\end{example}

\begin{example}
{\rm 
The cube  graph is the graph formed by the $8$ vertices and $12$ edges of a three-dimensional cube. For the cube graph $\min_S \rho(S)= 1/4$ so $b(G)=1$.
} \endproof
\end{example}

Due to this close connection to the sparsest cut, we can now conclude that computing an $\ell_1$-Fiedler vector is a computationally hard problem. 

\begin{corollary}
 \label{cor:complexity} 
   The computation of $b(G)$ and a corresponding $\ell_1$-Fiedler vector is {\em NP}-hard. 
\end{corollary}
\begin{proof}
 It was shown in \cite{BonsmaEtAl12} that computing min$\{\rho(S): \emptyset \subset S \subset V\}$ is {\em NP}-hard. Therefore, due to Theorem \ref{thm:main2}, the desired conclusion follows.  
\end{proof}

As mentioned, in \cite{BonsmaEtAl12} it is shown that the sparsest cut is {\em NP}-hard. The paper further contains a number of results showing that the problem is polynomially solvable for certain classes of graphs. For instance, this applies to graphs with bounded treewidth. Furthermore,  for cactus graphs, i.e., connected graphs where each edge is in at most one cycle, there is an algorithm for finding a sparsest cut which is linear in $n$, the number of vertices. This class includes trees, which we return to below. 
In \cite{AroraEtAl10} an efficient approximation algorithm for the sparsest cut was established, and it gives an $O(\sqrt{n})$ approximation (to the minimum value). 
In  \cite{Vazirani03} there is whole chapter devoted to the sparsest cut problem. Here the connection to a certain linear programming problem is shown. This is a multicommodity network flow problem where the goal is to maximize throughput for a given set of demands given by origin/destination (OD-) pairs and the corresponding flow demand. The special case where all demands are 1 and every pair is an OD-pair leads to an upper bound of the throughput which is the edge density. Therefore the minimum edge density, and a sparsest cut, corresponds to a bottleneck of the multicommodity flow problem. In this approach linear programming duality is combined with some basic results on embeddability of $\ell_1$-metric spaces. This  leads to a very important approximation algorithm with approximation error $O(\log n)$; for the details we refer to  \cite{Vazirani03} or \cite{Trevisan17}. 

Sparsest cuts are used in applications. For instance, in \cite{Wang03} sparsest cuts (which is called ratio cuts there) are used in image segmentation. They also show that  the sparsest cut problem is polynomially solvable in planar graphs, which is of interest in image analysis. In \cite{Osher03} one considered a segmention (or decomposition) problem in image analysis, using a model based on partial differential equations and total variation norm ($L_1$-norm).

\section{Comparison with other parameters}
\label{sec:compare}

This section establishes bounds on $b(G)$ in terms of other parameters. 

By combining Theorem \ref{thm:main2} and Theorem \ref{thm:Mohar} we now obtain the following bounds on $b(G)$. 

\begin{corollary} 
Let $G$ be a graph of order $n$. Then 
\[
    (1/2) a(G) \leq b(G)  \leq  (1/2)  \lambda_{1},
\]
where $a(G)$ is the algebraic connectivity of $G$ and $\lambda_1$ is the largest eigenvalue of $L(G).$
\end{corollary}

Let $d_{\rm min}(G)$ denote the smallest degree of a vertex in $G$. 

\begin{corollary}
\label{cor:bounds_alpha}
 For any graph $G$ the following bounds on $b(G)$ hold 
 \[
     \min_S \xi(S) \le b(G) \le \frac{n}{2(n-1)} d_{\rm min}(G).
 \]
\end{corollary}
\begin{proof}
 Let $S_1$ and $S_2$ be such that the minimum in (\ref{eq:alpha_exp}) is attained (so $S_2=V \setminus S_1$). Then 
 \[
    \min_S \xi(S) \le \min\{\xi(S_1), \xi(S_2)\} \le (1/2)(\xi(S_1)+\xi(S_2))= b(G).
 \]
Next, let $v$ be a vertex with smallest degree in $G$, so $d_v=d_{\rm min}(G)$, and let $S=\{v\}$. Then $\rho(S)=d_v/(n-1)$ so, by Theorem \ref{thm:main}, an upper bound on $b(G)$ is $\frac{n}{2(n-1)} d_{\rm min}(G)$. 
\end{proof}

\medskip
Let $mc(G)$ denote the cardinality of a minimum cut in $G$, i.e.,
\[
    mc(G)=\min\{|\delta(S)|: \emptyset \subset S \subset V\}.
\]
The next result is proved similar to the upper bound in the previous corollary, by letting $S$ be such that $\delta(S)$ is a minimum cut. 

\begin{corollary}
\label{cor:bounds_alpha2}
 For any graph $G$ 
 \[
      b(G) \le \frac{n}{2s(n-s)} mc(G).
 \]
 where $s=|S|$ and $\delta(S)$ is a minimum cut in $G$.
\end{corollary}
This bound is of interest because a minimum cut may be found efficiently (polynomial time) by simple greedy  algorithms. Moreover, such a minimum cut may be the starting point of algorithms for computing approximations to $b(G)$.

In order to compare $b(G)$ to the algebraic connectivity $a(G)$ we need a well-known property of  norms.

%
%
%
%

Let $m=|E|$ be the number of edges in $G$. 

\begin{theorem}
 \label{thm:a_alpha_1} 
 For every graph $G$
 \begin{equation}
 \label{eq:a_alpha}
      b(G)  \le \sqrt{m \cdot a(G)}.  
 \end{equation}
\end{theorem}
\begin{proof}
Let $x'$ be a Fiedler vector of unit length, so $\sum_{uv \in E} (x'_u-x'_v)^2=a(G)$ and $\sum_v x'_v=0$, $\|x'\|_2=1$. Let $x^*=tx'$ where $t=1/\|x'\|_1$. Then $\|x^*\|_1=1$, $\sum_j x^*_j=0$ and  $t =1/\|x'\|_1\le 1/\|x'\|_2=1$ as, in general,  $\|z\|_2 \le  \|z\|_1$ for $z \in \mb{R}^n$. Therefore
\[
 \begin{array}{ll} \vspace{0.2cm}
  b(G) &=\min\{\sum_{ij \in E} |x_i-x_j|: \sum_j x_j=0, \; \|x\|_1=1\} \\ \vspace{0.2cm}
     & \le  \sum_{ij \in E} |x^*_i-x^*_j|  \\ \vspace{0.2cm}
     & =  t \sum_{ij \in E} |x'_i-x'_j| \cdot 1  \\ \vspace{0.2cm}
     & \le  (\sum_{ij \in E} (x'_i-x'_j)^2)^{1/2} \cdot(\sum_{ij \in E} 1^2)^{1/2} \\ \vspace{0.2cm}
     &=  \sqrt{m \cdot a(G)}
 \end{array}
\]
where the last inequality is obtained from the Cauchy-Schwarz inequality.
\end{proof}

\section{Examples and special graphs}
\label{sec:special}

This section contains  examples and results for special graphs.

The next example considers the complete graph, which is an extreme case in terms of the parameter $b(G)$. 
 
\begin{example} 
{\rm 
Consider $K_n$ the complete graph of order $n$. Let $S \subset V$ where $S \not = \emptyset, V$. 
Define  $s=|S|$. Then 
\[
   \rho(S)=\frac{s(n-s)}{s(n-s)}=1.
\]
which is independent of $S$! So, by Theorem \ref{thm:main2}, $b(K_{n})=\frac{n}{2}$. Note that $K_n$ is an $(n-1)$-regular graph. It is known that $a(K_n) = n$. Thus $ b(K_n) < a(K_{n})$. \endproof
}
\end{example}

\begin{example}
{\rm 
Consider the cycle $C_n$, with $n\geq 4$ even. For all $S$, such that $|S|=p$, then $\rho(S)= \frac{2}{p(n-p)}.$ In this case $\min_S \rho(S)= \frac{8}{n^2}$, so $b(G)=4/n$. If $n\geq 3$ and $n$ is odd, then  $\min_S \rho(S)= \frac{2}{\lfloor n/2 \rfloor \cdot \lceil n/2 \rceil}$ and $b(G) =\frac{n}{\lfloor n/2 \rfloor \cdot \lceil n/2 \rceil }$.
}\endproof
\end{example}

A  wheel graph $ W_n$, with $n \geq 4$, is a graph formed by connecting a single vertex (central vertex) to all vertices of a cycle.

\begin{example}
{\rm
Let $n \geq 4$. Then $b(W_n) = \frac{n}{n-2}$.

In fact,
\begin{enumerate}
\item Consider all subsets $S$, such that $|S|=1$, say $S=\{v\}$. Then, two situations can occur:
\begin{itemize}
\item $v$ is the central vertex. Then,  $\rho(S)= \frac{n-1}{n-1}=1.$
\item  $v$ is not the central vertex. Then,  $\rho(S)= \frac{3}{n-1}.$
\end{itemize}
\item Consider all remaining subsets $S$, such that $|S|=i \geq 2$. Again, two situations can occur:
\begin{itemize}
\item $S$ does not contain the central vertex. Then,  $\rho(S)= \frac{i+2}{i(n-i)}.$
\item  $S$ contains the central vertex. Then $\rho(S)= \frac{n-i+2}{i(n-i)}.$
\end{itemize}
\end{enumerate} 
With some calculus we see that $\min_S \rho(S)$ is equal to $2/(n-2)$, so $b(G)=n/(n-2)$.
%
} \endproof
\end{example}

For trees we can find a sparsest cut analytically, as described next. 
Let $T=(V,E)$ be a tree, and let $e =uv \in E$. Then $T \setminus \{e\}$ consists of two disjoint trees, $T^e_u$ and $T^e_v$, where $T^e_u=(V^e_u,E^e_u)$ contains $u$, $T^e_v=(V^e_v,E^e_v)$ contains $v$ and $V^e_u \cup V^e_v=V$. We call $e=uv$ a {\em center edge} if $||V^e_u|-|V^e_v||$ is smallest possible and we let $\Delta(T)$ denote this minimum value. For instance, let $T$ be the path $P_n$. if $n=|V|$ is even, then $\Delta(P_n)$ is 0 , and there exist only one center edge. If $n$ is odd, then $\Delta(P_n)$ is 1 and there are  two center edges. If $T$ is the star $S_n$,  then $\Delta(S_n)=n-2$ and all the edges of $S_n$  are center edges. 

\begin{theorem}
 \label{thm:tree} 
  Let $T=(V,E)$ be a tree. Then there exists a center edge $e =uv \in E$  such that  the cut  $\delta(V^e_u)$ is a sparsest cut, and 
\[
  b(T)=(1/2)\big(1/|V^e_u|+1/|V^e_v|\big).
\]
\end{theorem}
\begin{proof}
 We apply Theorem \ref{thm:main2}, so a sparsest cut must be of the form $C=\delta(V^e_u)$ for some edge $e=uv$. The corresponding edge  density of the cut $C$ is 
 \[
    \rho(V^e_u)=\frac{1}{|V^e_u||V^e_v|},
 \]
 and it is easy to compute that this is minimized precisely when $||V^e_u|-|V^e_v||$ is smallest possible, i.e., when $e$ is a center edge and therefore 
 \begin{equation}\label{b2}
     b(G)=(n/2)\rho(V^e_u)=(1/2)\big(1/|V^e_u|+1/|V^e_v|\big).
 \end{equation}
\end{proof}

A {\em substar} $S$ in a tree $T$ is a vertex-induced subgraph which is a star, i.e., it consists of edges sharing a common vertex. 

\begin{corollary}
 \label{cor:tree2} 
   Let $T=(V,E)$ be a tree. Then the set of center edges is a substar. 
\end{corollary}
\begin{proof}
 We prove this by contradiction. So, assume there are two center edges $e=uv$ and $e'=u'v'$ that are disjoint (no common vertex). Since $T$ is connected there is a path $P$ between a vertex in $e_1$ and a vertex in $e_2$, say that $v$ and $v'$ are closest with this property. 
 
 {\em Claim: $|V^e_u|\le |V^e_v|$, and $|V^{e'}_{u'}|\le |V^{e'}_{v'}|$.} 
 
 Proof of Claim: Otherwise $|V^e_u| > |V^e_v|$, and $V^{e'}_{u'} \subset V^e_v$ (strict subset) 
 so $|V^{e'}_{v'}|>|V^e_u| > |V^e_v|>|V^{e'}_{u'}|$. This contradicts that $e'$ is a center edge (because the two numbers to the left and right in these inequalities are further apart that the two in the middle). This proves the claim.
 
Next, from the Claim, we conclude that $|V^e_u|=|V^{e'}_{u'}|$ because both $e$ and $e'$ are center edges (otherwise these two numbers would have different distance to $n/2$; see the discussion on center edge above). Let $e^*=vw$ be the edge in the path $P$ that is incident to $v$. Then $e^*$ is different from $e$ and $e'$. 
Moreover, $|V^{e^*}_v| \ge n/2$, otherwise $|V^{e^*}_v|$ would be closer to $n/2$ than $|V^e_u|$; contradicting that $e$ is a center edge. This implies that $|V^{e^*}_w| \le n/2$ and then 
\[
    |V^{e'}_{u'}|  < |V^{e^*}_w| \le n/2,
\]
contradicting that $e'$ is a center edge.  Therefore, such an edge $e^*$ does not exist, meaning that $e$ and $e'$ must be incident.
\end{proof}

\begin{corollary}
 \label{cor:path} 
  Let $P_n$ be the path with $n$ vertices. 
  
  If $n$ is even, then $b(P_n)=2/n$ and an $\ell_1$-Fiedler vector has the first $n/2$ components equal to $1/n$ and the last $n/2$ components are equal to $-1/n$.  
  
  If $n$ is odd, then $b(P_n)=2n/(n^2-1)$.  An $\ell_1$-Fiedler vector has the first $(n-1)/2$ components equal to $1/(n-1)$ and the remaining components are equal to $-1/(n+1)$. 
  
\end{corollary}

\begin{proof}
 This follows from Theorem \ref{thm:tree}. Let $n$ be even. Taking $|V^e_u|= p,$ and $|V^e_v|= n-p,$ we have $\delta(P_n) = \frac{1}{p(n-p)}$, where the minimum of this function is attained for $p= n/2$. Then $b(P_n)=2/n$. From the proof of Theorem \ref{thm:main} the $\ell_1$-Fiedler vector has $p= n/2$ entries equal to 
 $\frac{1}{2p}=\frac{1}{n}$ and $n/2$ entries equal to $-\frac{1}{2(n-p) }= -\frac{1}{n}$.  The case when $n$ is odd is analogous.
\end{proof}

\begin{corollary}
 \label{cor:star} 
  Let $S_n$ be the star with $n$ vertices. Then 
  \[
     b(S_n)= \frac{1}{2} + \frac{1}{2(n-1)}.
  \]
 Moreover, the entries of an $\ell_1$-Fiedler vector are $\frac{1}{2}$ for the center vertex and  $- \frac{1}{2(n-1)}$ for the remaining vertices.
\end{corollary}
\begin{proof}
 This follows from Theorem \ref{thm:tree} and Theorem \ref{thm:main}.
\end{proof}

 %
 %
 

Let $\mathcal{C}$ be the set of center edges  of a given tree. We next look at some examples concerning specific classes of trees.

\begin{example} 
{\rm 
 Consider a {\em broom-tree} $B(\ell, n-\ell)$ consisting of a path $v_1, v_2, \ldots, v_\ell$ and additional vertices $v_{\ell+1}, v_{\ell+2}, \ldots, v_n$ each attached to $v_\ell$. Here $e_{ij}$ denotes the edge $v_i v_j$ or simply $ij$. There are two cases to discuss. 
 \begin{enumerate}
 
 \item  If $\ell \leq  n-\ell$, i.e., $\ell \le n/2$,  then there is a  unique center edge, and it is $e=e_{\ell-1, \ell}$. So $\mathcal{C}= \{e\}$ and 
 \begin{equation}\label{b}
 b(B(\ell, n-\ell))= \frac{1}{2} \big(\frac{1}{\ell-1} + \frac{1}{n-\ell+1}\big)= \frac{n}{2 (\ell-1)(n-\ell+1)}.
 \end{equation}
 
 In fact, let us consider an edge $e_{k, k+1}$ with $k= 1, 2, \ldots,  \ell-1$, and the function $f(k) = \frac{1}{2} ( \frac{1}{k} +\frac{1}{n-k}).$ This function is decreasing and its minimum is attained for $k=n-1$. 
 Additionally, for any edge $e_{\ell, \ell+ p}$ for $p=1, \ldots, n-p,$ the expression in $(\ref{b2})$ is $(1/2)\big(1+\frac{1}{n-1}\big)$ that is clearly greater or equal than the expression in $(\ref{b}).$ Thus $b(B(\ell, n-\ell))$ is as in $(\ref{b}).$
 \item  If $\ell > n-\ell$, i.e., $\ell > n/2$,  then two subcases can occur. 
\begin{enumerate}
\item $n$ is even, say $n=2k$ for some integer $k$.  

Then, $e= e_{k, k+1}$ is the unique center edge, so  $\mathcal{C}= \{e\}.$ In this case we have 
\[
 |V^{e}_{v_k}|=  |V^{e}_{v_{k+1}}|= \frac{n}{2},
 \]
and 
\[
 b (B(\ell, n-\ell)) =  \frac{1}{2}\big(\frac{1}{n/2} + \frac{1}{n/2} \big)= \frac{2}{n}= \frac{1}{k}.
 \]

\item $n$ is odd, say $n=2k+1$ for some integer $k$. 

\begin{enumerate}
\item If $\ell -k =1$, then there exists an unique center edge $e= e_{\frac{n-1}{2}, \frac{n+1}{2}}$ and $\mathcal{C}= \{e\}.$ 
Then 
\begin{equation} \label{b3}
 b (B(\ell, n-\ell)) =\frac{1}{2}\left ( \frac{1}{k }+ \frac{1}{n- k } \right) = \frac{1}{n-1}+ \frac{1}{n+1}=2n/(n^2-1). \\
 \end{equation}
 
 \item Suppose now that $\ell -k >1.$  Then we have two center edges
 $e= e_{k,k+1}$ and $e^{\star} = e_{k+1,k+2},$ that is  $e= e_{\frac{n-1}{2},\frac{n+1}{2}}$ and $e^{\star} = e_{\frac{n+1}{2},\frac{n+3}{2}}.$  In this case $\mathcal{C} = \{e, e^{\star}\}$ and $b (B(\ell, n-\ell))$ is the same value as $(\ref{b3}).$
 
 \end{enumerate}
 \end{enumerate}
 \end{enumerate}
 } \endproof
 \end{example}
  
 A pendant edge in a graph is an edge where one of its vertices has degree $1.$
\begin{example} 
{\rm 
Let $n_1, n_2, \ldots, n_k$ be positive integers. The {\em starlike tree} $S(n_{1}, n_{2}, \ldots, n_{k})$ is the tree that results from the stars $S_{n_{1}}, S_{n_{2}}, \ldots, S_{n_{k}}$ by connecting their centers to an extra vertex $v$, see \cite{Arnold}.
 Let $n=\sum_{i=1}^{k }n_{i}+1$ and define  $n_M=\max_{i=1,2,\dots,k}n_i$.  Then, 
  \[ 
     b(S(n_{1}, n_2, \ldots, n_{k}))=\frac{n}{2 n_{M} (n-n_{M})}.
  \]
  If $|S_p|>|S_i|$ ($i \not = p$), then there is a unique center edge, which is  $v v_{\ell}$.  Otherwise the maximum of $|S_i|$ is attained for more than one $i$, and the center edges constitute a substar connecting $v$ to the center of those stars $S_i$.
 }
  \endproof
  \end{example}
   




\bigskip

Finally, in this section,  we give a computational example for a graph $G$. It  illustrates the partitions obtained based on the Fiedler vector and the $\ell_1$-Fiedler vector, respectively.

\begin{example}
{\rm 
 
 In this example we generated some ``random'' points in the plane and constructed edges between points that were closer than some given distance. This gave a graph $G$. The graph contained 15 vertices and 69 edges. We wanted to compare the partitions obtained from the $\ell_2$ (Fiedler) versus $\ell_1$ (sparsest cut) graph smoothing approach. 
 
 In the $\ell_2$ approach we obtained a cut with 22 edges and the partition contained 7 and 8 vertices, respectively. In the $\ell_1$ approach we obtained a cut with 12 edges and the partition contained 3 and 13 vertices, respectively. The corresponding edge densities were $0.39$ and $0.33$. So the sparsest edge density is $0.33$.  Note that the sparsest cut is quite unbalanced, but has very few edges compared to the cut in the Fiedler partition. For certain graphs this may happen, but graphs are so different that we will not make any general claims on properties of these solutions.  
\begin{figure}
\centering
\includegraphics[width=80mm]{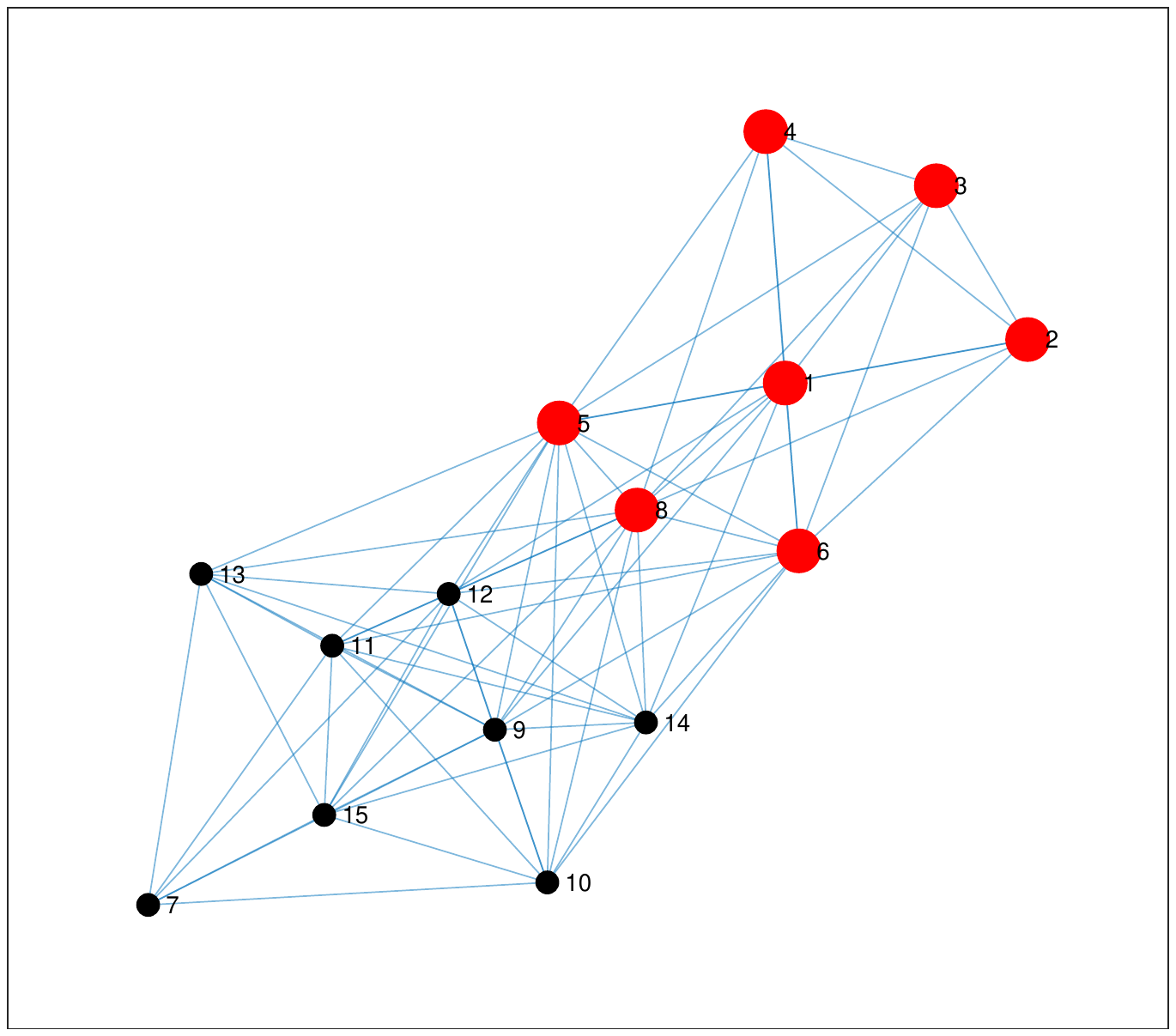}
\includegraphics[width=80mm]{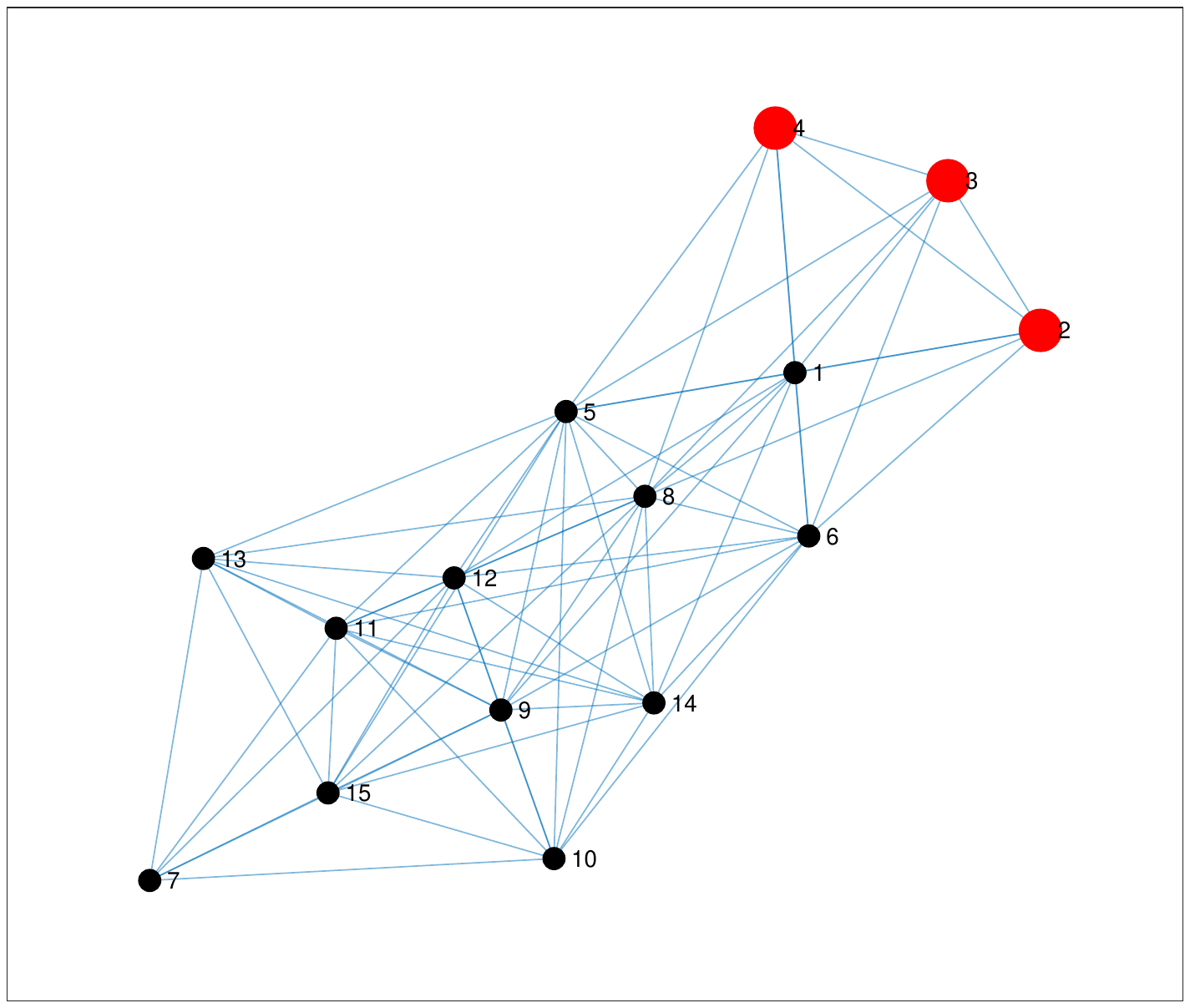}
\vspace{-2cm}
\caption{Graph smoothing partitions: $\ell_2$ (left) and $\ell_1$ (right).}
\label{fig:example_smoothing}
\end{figure}
} \endproof
\end{example}

Concerning the previous example note that each of the two solutions are optimal in the respective optimization problems, so comparing partitions is a bit artificial. However, both methods may be used for partitioning, and which one is the better is hard to say in general. It depends on the underlying application and the type of graph considered. Computationally, the Fiedler vector can be computed fast, while for larger graphs we need heuristics for finding approximate sparsest cuts. On the other hand, a sparse or sparsest cut is a reasonable notion and can be analysed for specific graph classes, while the Fiedler vector relies on an eigenvector which is not easy to give a direct combinatorial interpretation.
Finally, for clustering it might be possible to combine these two methods, possibly also with  some of the many other approaches known for graph clustering.

\section{Graph smoothing in $\ell_{\infty}$-norm and concluding remarks}

We conclude the paper with some remarks. First, it also makes sense to consider the graph smoothing problem in other norms, such as the $\ell_{\infty}$-norm (given by $\|x\|_{\infty}=\max_i |x_i|$). This leads to the {\em $\ell_{\infty}$-graph smoothing} problem
\begin{equation}
 \label{eq:smooth_inf}
     \gamma(G)=\min\{\max_{uv \in E} |x_u-x_v|: \sum_{v \in V} x_v=0, \; \|x\|_{\infty}=1\}.
\end{equation}
In contrast to the $\ell_1$-graph smoothing problem we can solve $(\ref{eq:smooth_inf})$ efficiently, for general graph $G$.

\begin{theorem}
 The $\ell_{\infty}$-graph smoothing problem $(\ref{eq:smooth_inf})$ can be solved in polynomial time, using linear programming.
\end{theorem}
\begin{proof}
We rewrite problem $(\ref{eq:smooth_inf})$ using a construction rather similar to what we did for the  $\ell_1$-graph smoothing problem in (\ref{eq:LP-smooth}). For each $k \le n$ consider the  linear programming problem LP$(k)$: 
\begin{equation}
 \label{eq:LP_inf}
 \begin{array}{lrlll} \vspace{0.1cm}
 \mbox{\rm minimize}  & y \\ \vspace{0.1cm}
    &       x_i - x_j  &\le &y  & (ij \in E),\\ \vspace{0.1cm}
     &      -(x_i - x_j) &\le &y & (ij \in E),\\ \vspace{0.1cm}
     & \sum_{j=1}^n x_j&= &0, \; \\ \vspace{0.1cm}
      &  \multicolumn{3}{r}{-1 \le x_j \le 1} & (j \le n),  \\ \vspace{0.1cm}
      &  \multicolumn{3}{r}{\mbox{\rm $x_k=1$.}}
 \end{array}
\end{equation}
Here the variables are  $x=(x_1, x_2, \ldots, x_n)$ and $y$. Note that  in every optimal solution of (\ref{eq:LP_inf}) $y=\max_{uv \in E} |x_u-x_v|$, and $\|x\|_{\infty}=1$. It follows that LP$(k)$ solves the $\ell_{\infty}$-graph smoothing problem under the additional restriction that $x_k=1$. (Due to symmetry, $-x$ satisfies the first four constraints when $x$ does, so restricting one component to $-1$ is not needed.) So, by solving LP$(k)$ for every $k \le n$, and taking the minimum $y$ found, we solve the $\ell_{\infty}$-graph smoothing problem. Since LP problems can be solved in polynomial time, the resulting algorithm is a polynomial-time algorithm.
\end{proof}

\begin{proposition}
 Consider the path $P_n$ with $n$ vertices. Then the minimum value of the $\ell_{\infty}$-graph smoothing problem $(\ref{eq:smooth_inf})$ is $\gamma=\gamma(P_n)=2/(n-1)$ and an optimal solution is 
 \[
    x=(1, 1-\gamma, 1-2\gamma, \ldots, 1-(n-2)\gamma, -1).
 \]
\end{proposition}
\begin{proof}
 For the specified $x$ we have $|x_i-x_{i+1}|=\gamma$ for $i\le n-1$, so $\max_i |x_i-x_{i+1}|=\gamma$. Moreover, $\sum_i x_i=0$ and $\|x\|_{\infty}=1$. Assume there is an $z=(z_1, z_2, \ldots, z_n)$ satisfying $\sum_i z_i=0$,  $\|z\|_{\infty}=1$ and  $\max_i |z_i-z_{i+1}|<\gamma$. By symmetry we may assume that some component of $z$ is 1, say $z_k=1$. Then, if $k<n$,  $|z_k-z_{k+1}|<\gamma$, so $z_{k+1}>1-\gamma$. Similarly, $z_{k+2}>1-2\gamma$, and in general $z_s>1-|s-k|\gamma$. But then it is easy to see that 
\[
  \sum_i z_i > \sum_i x_i=0,
\]
which contradicts that $\sum_i z_i=0$. Thus, the optimal value of  $(\ref{eq:smooth_inf})$ is $\ge \gamma$, and therefore equal to $\gamma$, as desired.
\end{proof}

\begin{example}
{\rm 
Consider the path $P_6=v_1, v_2, v_3, v_4,v_5,v_6$, with corresponding variables $x_i$ ($ i \le 5$). Then $\gamma=2/5$ and an optimal solution of the $\ell_{\infty}$-graph smoothing problem is
\[
    x=(1, 3/5,1/5,-1/5, -3/5, -1).
\]
Thus consecutive components of $x^*$ differ in absolute value by $y=2/5$. 
} \endproof
\end{example}

One can also solve $(\ref{eq:smooth_inf})$ explicitly for stars, and some other graphs. It is ongoing work to investigate the $\ell_{\infty}$-graph smoothing problem further and to relate to the  $\ell_1$- and $\ell_2$-graph smoothing problems. Moreover, it is interesting to see if the corresponding cuts obtained from the signs in an optimal $x^*$ are useful in partitioning problems.

\medskip
Finally, we remark that the main contribution of the present paper was to investigate a variant of the optimization (variational) characterization of algebraic connectivity, by changing into the $\ell_1$-norm. We showed strong optimality properties that are similar to the Fiedler theory for algebraic connectivity. Also, we showed that optimal solutions correspond to sparsest cuts which gives a new way to view these combinatorial objects. We believe that further work on similar optimization problems, in different norms, would be interesting. This includes to combine the different approaches into useful algorithms for important applications in clustering problems in graphs.


\begin{thebibliography}{99}


\bibitem{Alon} N.~Alon,  Eigenvalues and expanders,  {\em Combinatorica} 6 (1986) 83--96.

\bibitem{AlonandMilman}  N.~Alon, A.~Milman, Isoperimetric inequalities for graphs, and superconcentrators, {\em J. Combin. Theory} Ser. B 38 (1985), 73--88.

\bibitem{EALCGD19}  E.~Andrade, L.~Ciardo,  G.~Dahl,  Combinatorial Perron parameters for trees, \textit{Linear Algebra Appl.}, 566 (2019), 138--166.

\bibitem{EAGD17} E.~Andrade,  G.~Dahl,  Combinatorial Perron values of trees and bottleneck matrices, \textit{Linear Multilinear Algebra} 65(12) (2017), 2387--2405.

\bibitem{AroraEtAl10} S.~Arora, E.~Hazan,  S.~Kale, 
$O(\sqrt{n})$ approximation to  sparest cut in $\tilde{O}(n^2)$ time, 
{\em SIAM J. Comput.}  39 (5) (2010), 1748--1771. 

\bibitem{BonsmaEtAl12} P.~Bonsma, H.~Broersma, V.~Patel, A.~Pyatkin, The complexity of finding uniform sparsest cuts in various graph classes, {\em J.  Discrete Algorithms}, 14 (2012), 136--149.

\bibitem{CiardoZivny23}
L.~Ciardo, S.~\v{Z}ivn\'{y},  
Approximate Graph Colouring and Crystals, 
Proceedings of the 2023 Annual ACM-SIAM Symposium on Discrete Algorithms (SODA)
(2023), 2256--2267. 

\bibitem{Cheeger} J.~Cheeger, A lower bound for the smallest eigenvalue of the Laplacian, In: R.C. Gunnig. eds.,
{\em Problems in Analysis}, Princeton Univ. Press, Princeton, NJ, 1970,  195--199.

\bibitem{CottleDantzig68} R.W.~Cottle, G.B.~Dantzig,  Complementary pivot theory of mathematical programming, {\em Linear Algebra Appl.} 1, (1968), 103--125. 


\bibitem{Hoory} S.~Hoory, N.~Linial, A.~Wigderson, Expander graphs and their applications, {\em Bull. Amer. Soc. } 43 (4) (2006), 439--561.


\bibitem{Fallat_Kirkland_Pati}
S.~Fallat,  S.~Kirkland,  S~Pati, On graphs with algebraic connectivity equal to minimum edge density, \textit{Linear Algebra Appl.}, 373 (2003), 31--50.



\bibitem{Fiedler_alg_conn}
M.~Fiedler,  Algebraic connectivity of graphs,  \textit{Czech. Math. J.}  23(2) (1973), 298--305.

\bibitem{Fiedler2} M.~Fiedler,  A property of eigenvectors of nonnegative symmetric matrices and its application to graph theory,  \textit{Czech. Math. J.}  25 (1975), 619–633.



\bibitem{HornJohnson13} 
R.A.~Horn,  C.R.~Johnson,  \textit{Matrix Analysis}, Second Edition,  Cambridge University Press, New York, 2013. 


\bibitem{Kim}S.~Kim, S.~Kirkland, Fiedler vectors with unbalanced sign patterns, {\em Czech. Math.}  J., 71 (146) (2021), 1071--1098.




\bibitem{Arnold} A.~Knopfmacher, R.F.~Tichy, S.~Wagner,  V.~Ziegler,  Graphs, partitions and Fibonacci numbers. 
{\em Discrete Appl. Math.} 155(10) (2007), 1175--1187.




\bibitem{Luxburg} U. von~Luxburg, A tutorial on spectral clustering, {\em Statistics and Computing} 17: 4 (2007), 395--416.


\bibitem{Mohar} B.~Mohar,  Laplace eigenvalues of graphs -- a survey, {\em Discrete Math.} 109 (1992) 171--183.

\bibitem{Mohar2} B.~Mohar, Isoperimetric numbers of graphs,  {\em J. Combin. Theory}, Ser. B 47 (1989) 274-291.


\bibitem{Molitierno11}  J.J. Molitierno,
\newblock {\em Applications of Combinatorial Matrix Theory to Laplacian Matrices of Graphs}, 
\newblock Boca Raton,  CRC Press,  2012.

\bibitem{Osher03} 
S.~Osher, A.~Sol\'{e}, L.~Vese, Image decomposition and restoration using total variation minimization and the $H^{-1}$ norm, {\em SIAM Multiscale Model. Simul.}
 1 (3) (2003),  349--370.


\bibitem{Schaeffer} 
S.~Schaeffer, Graph clustering, {\em Comput. Sci. Rev.} 1 (1) (2007), 27--64.

\bibitem{Spielman10} 
 D.A.~Spielman, 
 Algorithms, Graph Theory, and Linear Equations in Laplacian Matrices, 
Proceedings of the International Congress of Mathematicians,  Hyderabad, India, 2010.


\bibitem{Spielman19} D.A.~Spielman, Spectral and Algebraic Graph Theory,  Lecture notes, Yale University,   2019.


\bibitem{Trevisan17} L.~Trevisan, Lecture Notes on Graph Partitioning,
Expanders and Spectral Methods, Lecture notes, University of California, Berkeley, 2017.

\bibitem{Urschel} C.~Urschel, L.T.~Zikatanov, On the maximal error of spectral approximation of graph
bisection, {\em  Linear Multilinear Algebra} 64 (2016), 1972--1979.

\bibitem{Vazirani03} V.V.~Vazirani, {\em Approximation Algorithms}, Second printing, Springer, Berlin Heidelberg, 2003. 

\bibitem{Wang03} S.~Wang, J.M.~Siskind, Image segmentation with ratio cut, 
{\em IEEE Trans. Pattern Anal. Mach. Intell. }, 25 (6) (2003), 675--690.

\end{thebibliography}
\end{document}